\documentclass[11pt,oneside]{amsart}
\usepackage[cp1250]{inputenc}
\usepackage{mathtools}
\usepackage{amstext}
\usepackage{amsthm}
\usepackage{amssymb}
\usepackage[unicode=true,
 bookmarks=false,
 breaklinks=false,pdfborder={0 0 1},backref=false,colorlinks=false]
 {hyperref}

\makeatletter
\numberwithin{equation}{section}
\numberwithin{figure}{section}
\theoremstyle{plain}
\newtheorem{thm}{\protect\theoremname}
  \theoremstyle{plain}
  \newtheorem{lem}[thm]{\protect\lemmaname}
  \theoremstyle{remark}
  \newtheorem{rem}[thm]{\protect\remarkname}
  \theoremstyle{plain}
  \newtheorem{prop}[thm]{\protect\propositionname}

\@ifundefined{date}{}{\date{}}
\usepackage{a4wide}
\usepackage{sidenotes}

\makeatother

  \providecommand{\lemmaname}{Lemma}
  \providecommand{\propositionname}{Proposition}
  \providecommand{\remarkname}{Remark}
\providecommand{\theoremname}{Theorem}

\begin{document}

\title[Quantum symmetries on noncommutative complex spheres]{Quantum symmetries on noncommutative complex spheres with partial
commutation relations}

\author{Simeng Wang}

\address{Saarland University, Fachbereich Mathematik, Postfach 151150, 66041
Saarbrücken, Germany}

\email{wang@math.uni-sb.de}

\date{\today}
\begin{abstract}
We introduce the notion of noncommutative complex spheres with partial
commutation relations for the coordinates. We compute the corresponding
quantum symmetry groups of these spheres, and this yields new quantum
unitary groups with partial commutation relations. We also discuss
some geometric aspects of the quantum orthogonal groups associated
with the mixture of classical and free independence discovered by
Speicher and Weber. We show that these quantum groups are quantum
symmetry groups on some quantum spaces of spherical vectors with partial
commutation relations.
\end{abstract}

\maketitle

\section{Introduction}

This paper introduces a new class of noncommutative spheres and discusses
the associated quantum symmetry groups. The quantization of classical
spheres was initiated in the work of Podl\'{e}s \cite{podles87sphere,podles95spheresym}.
The theory of various noncommutative spheres and their quantum symmetries
has been then remarkably developed in the past decades (see for example
\cite{conneslandi01sphere,connesduboisviolette02sphere,goswami09qsymsphere,banicagoswami10qsymsphere,banica15complexsphere,banica16spheredual,banica17spherenote}
and references therein).

In a recent work \cite{speicherweber16epsilinqg}, Speicher and Weber
introduced a new class of noncommutative spheres with partial commutation
relations, and computed the corresponding quantum symmetry group.
This also leads to new versions of quantum orthogonal groups which
do not interpolate between the classical and universal versions of
orthogonal groups. 

In this note we will continue the project proposed by \cite{speicherweber16epsilinqg}.
We will discuss the complex versions of noncommutative spheres with
partial commutation relations. We will compute the quantum symmetry
groups of these objects. Compared to the real case studied in \cite{speicherweber16epsilinqg},
the complex case involves more subtlety such as the mixture of normal
and non-normal generators. Similarly as in the real case, we obtain
new examples of quantum unitary groups with partial commutation relations.
We refer to Section 2 for all details.

On the other hand, we also answer some unsolved problems in \cite{speicherweber16epsilinqg}
regarding the real case. In \cite{speicherweber16epsilinqg}, by virtue
of the mixture of independences in quantum probability, some quantum
orthogonal groups with partial commutation relations are introduced.
However the geometric aspects of these quantum groups were not clear
in their work. In this note we will construct some quantum tuples
of noncommutative spheres so that the corresponding quantum symmetry
groups are exactly those studied in \cite{speicherweber16epsilinqg}.
The result will be given in Section 3.

\section{The noncommutative complex spheres and quantum symmetries}

In this section we let $\varepsilon=(\varepsilon_{ij})_{i,j\in\{1,\ldots,n\}}$
and $\eta=(\eta_{kl})_{k,l\in\{1,\ldots,n\}}$ be two symmetric matrices
with $\varepsilon_{ij}\in\{0,1\}$, $\varepsilon_{ii}=0$ and $\eta_{kl}\in\{0,1\}$.

\subsection{Noncommutative complex $(\varepsilon,\eta)$-spheres}

We consider the universal C{*}-algebra 
\[
C^{*}(x_{1},\ldots,x_{n}\mid\sum_{i=1}^{n}x_{i}^{*}x_{i}=\sum_{i=1}^{n}x_{i}x_{i}^{*}=1,x_{i}x_{j}=x_{j}x_{i}\text{ if }\varepsilon_{ij}=1,x_{i}^{*}x_{j}=x_{j}x_{i}^{*}\text{ if }\eta_{ij}=1).
\]
As an intuitive notation, we denote the above C{*}-algebra by $C(S_{\mathbb{C},\varepsilon,\eta}^{n-1})$
and we say that $S_{\mathbb{C},\varepsilon,\eta}^{n-1}$ is a \emph{noncommutative
complex $(\varepsilon,\eta)$-sphere}. Note that if all non-diagonal
entries of $\varepsilon$ and all entries of $\eta$ are $1$, then
we obtain the algebra $C(S_{\mathbb{C}}^{n-1})$ of continuous functions
over the complex sphere $S_{\mathbb{C}}^{n-1}\subset\mathbb{C}^{n}$.
If all entries of $\varepsilon$ and $\eta$ are $0$, we get the
Banica's free version of complex spheres in \cite{banica15complexsphere}.

Compared to the real spheres studied in \cite{speicherweber16epsilinqg},
we consider two matrices $\varepsilon$ and $\eta$ rather than one
in order to include the case where the generators $x_{i}$ are non-normal.
Note that the diagonal entries of $\eta$ are related to the normality
of the generators. We give the following remarks.
\begin{lem}
\label{lem:normal}\emph{(1) }Let $1\leq i\leq n$. If $x_{i}$ is
normal (in other words if $\eta_{ii}=1$), then for any $j$, 
\[
x_{i}x_{j}=x_{j}x_{i}\quad\text{iff}\quad x_{i}^{*}x_{j}=x_{j}x_{i}^{*}.
\]

\emph{(2)} Let $1\leq i\leq n$. If $\varepsilon_{ij}=\eta_{ij}=1$
for all $j\neq i$ with $\eta_{jj}=0$, then $x_{i}$ is normal.\end{lem}
\begin{proof}
The assertion (1) follows immediately from the Fuglede theorem (see
for example \cite[12.16]{rudin91functbook}). Let us prove the assertion
(2). Without loss of generality let us assume that $\eta_{ii}=0$
for $1\leq i\leq k$, $\eta_{ii}=1$ for $k+1\leq i\leq n$, and $\varepsilon_{1j}=\eta_{1j}=1$
for $2\leq j\leq k$, and show that $x_{1}$ is normal.

Note that 
\[
\sum_{i=1}^{k}x_{i}^{*}x_{i}=1-\sum_{i=k+1}^{n}x_{i}^{*}x_{i}=1-\sum_{i=k+1}^{n}x_{i}x_{i}^{*}=\sum_{i=1}^{k}x_{i}x_{i}^{*}.
\]
Then we set 
\[
a=\sum_{i=1}^{k}x_{i}^{*}x_{i}=\sum_{i=1}^{k}x_{i}x_{i}^{*},\quad b=\sum_{i=2}^{k}x_{i}^{*}x_{i},\quad c=\sum_{i=2}^{k}x_{i}x_{i}^{*}.
\]
We have
\[
x_{1}^{*}x_{1}+b=x_{1}x_{1}^{*}+c=a.
\]
So 
\begin{equation}
x_{1}^{*}x_{1}-x_{1}x_{1}^{*}=c-b.\label{eq:c-b}
\end{equation}
Since $\varepsilon_{1j}=\eta_{1j}=1$ for $2\leq j\leq k$, we see
that $x_{1}$ commutes with $b$ and $c$, and in particular by \eqref{eq:c-b},
\[
x_{1}(x_{1}^{*}x_{1}-x_{1}x_{1}^{*})=(x_{1}^{*}x_{1}-x_{1}x_{1}^{*})x_{1}.
\]
Then by Putnam's theorem (see for example \cite[Coroallary 2.2.10]{sakai91oadyn}),
$x_{1}$ is normal.
\end{proof}
By virtue of Lemma \ref{lem:normal}, we make the convention that
\begin{equation}
\varepsilon_{ij}=\eta_{ij},\quad\text{if }\eta_{ii}=1\text{ or }\eta_{jj}=1.\label{eq:convention eta}
\end{equation}
and for each $i$ with $\eta_{ii}=0$, there exists $j\neq i$ with
$\eta_{jj}=0$ such that
\begin{equation}
\varepsilon_{ij}=0\text{ or }\eta_{ij}=0.\label{eq:convention eta 2}
\end{equation}
For simplicity, we say that the pair $(\varepsilon,\eta)$ is \emph{regular}
if \eqref{eq:convention eta} and \eqref{eq:convention eta 2} hold.
According to the above lemma, for any non regular pair $(\varepsilon,\eta)$
we may always find a regular one which associates the same sphere.
So we will only consider the regular case.
\begin{rem}
\label{rem:normal}By the above lemma, if all the generators $x_{i}$
are normal ($\eta_{ii}=1$ for all $1\leq i\leq n$), the C{*}-algebra
$C(S_{\mathbb{C},\varepsilon,\eta}^{n-1})$ can be simply determined
by the entries of $\varepsilon$, that is,
\[
C(S_{\mathbb{C},\varepsilon,\eta}^{n-1})=C^{*}(x_{1},\ldots,x_{n}\mid x_{i}^{*}x_{i}=x_{i}x_{i}^{*},\sum_{i=1}^{n}x_{i}^{*}x_{i}=1,x_{i}x_{j}=x_{j}x_{i}\text{ if }\varepsilon_{ij}=1).
\]
However, if some generators $x_{i}$ are not normal, it is still worth
considering two different matrices $\varepsilon,\eta$ rather than
one since there exist non-trivial representations of $C(S_{\mathbb{C},\varepsilon,\eta}^{n-1})$
with $\varepsilon_{ij}\neq\eta_{ij}$. For instance, take
\[
\varepsilon=\begin{bmatrix}0 & 1\\
1 & 0
\end{bmatrix},\quad\eta=\begin{bmatrix}0 & 0\\
0 & 0
\end{bmatrix},
\]
then there exists a representation 
\[
\pi:C(S_{\mathbb{C},\varepsilon,\eta}^{1})\to\mathbb{M}_{4}(\mathbb{C}),\quad\pi(x_{1})=a,\pi(x_{2})=b
\]
such that $a$ and $b$ are not normal and 
\[
ab=ba(\neq0),\ ab^{*}\neq b^{*}a.
\]
Indeed, it suffices to take 
\[
a=\begin{bmatrix}0 & 0 & 0\\
0 & 0 & 0\\
1 & 0 & 0\\
 &  &  & \frac{\sqrt{2}}{2}
\end{bmatrix},\quad b=\begin{bmatrix}0 & 0 & 0\\
1 & 0 & 0\\
0 & 1 & 0\\
 &  &  & \frac{\sqrt{2}}{2}
\end{bmatrix}.
\]

\end{rem}

\subsection{Quantum symmetries}

Now we introduce the corresponding quantum groups. We refer to \cite{woronowicz87matrix,woronowicz1998note,timmermann08qgbook}
for any unexplained notation and terminology on compact matrix quantum
groups. Define the universal C{*}-algebra 
\[
C(U_{n}^{\varepsilon,\eta})=C^{*}(u_{ij},i,j=1,\ldots,n\mid u\text{ and }\bar{u}\text{ are unitary, }R^{\varepsilon}\text{ and }R_{*}^{\eta}\text{ hold}),
\]
where $R^{\varepsilon}$ are the relations
\[
u_{ik}u_{jl}=\begin{cases}
u_{jl}u_{ik} & \text{if \ensuremath{\varepsilon_{ij}=1} and \ensuremath{\varepsilon_{kl}=1}}\\
u_{jk}u_{il} & \text{if \ensuremath{\varepsilon_{ij}=1} and \ensuremath{\varepsilon_{kl}=0}}\\
u_{il}u_{jk} & \text{if \ensuremath{\varepsilon_{ij}=0} and \ensuremath{\varepsilon_{kl}=1}}
\end{cases},
\]
and $R_{*}^{\eta}$ are the relations
\begin{equation}
u_{ik}^{*}u_{jl}=u_{jl}u_{ik}^{*},\quad\text{if }\eta_{ij}=\eta_{kl}=1,\label{eq:11}
\end{equation}
\begin{equation}
u_{ik}^{*}u_{jl}=0,\ u_{ik}u_{jl}^{*}=0,\quad\text{if }\eta_{ij}=1,\eta_{kl}=0,k\neq l,\label{eq:zero l}
\end{equation}
\begin{equation}
u_{ik}^{*}u_{jl}=0,\ u_{ik}u_{jl}^{*}=0,\quad\text{if }\eta_{ij}=0,\eta_{kl}=1,i\neq j,\label{eq:zero r}
\end{equation}
\begin{equation}
u_{ik}^{*}u_{jk}=u_{il}^{*}u_{jl}=u_{jk}u_{ik}^{*}=u_{jl}u_{il}^{*},\quad\text{if }\eta_{ij}=1,\eta_{kk}=\eta_{ll}=0,\label{eq:eq kk}
\end{equation}
\begin{equation}
u_{ki}^{*}u_{kj}=u_{li}^{*}u_{lj}=u_{kj}u_{ki}^{*}=u_{lj}u_{li}^{*},\quad\text{if }\eta_{ij}=1,\eta_{kk}=\eta_{ll}=0.\label{eq:eq kk left}
\end{equation}
So for $\eta_{ij}=1$ we may define 
\begin{equation}
X_{ij}=u_{ik}^{*}u_{jk}=u_{jk}u_{ik}^{*},\quad Y_{ij}=u_{ki}^{*}u_{kj}=u_{kj}u_{ki}^{*},\label{eq:xij yij}
\end{equation}
where $k$ is any index satisfying $\eta_{kk}=0$. Note that this
definition does not depend on the choice of $k$ by virtue of \eqref{eq:eq kk}
and \eqref{eq:eq kk left}. 

We consider the comultiplication defined by
\[
\begin{array}{cccc}
\Delta: & C(U_{n}^{\varepsilon,\eta}) & \to & C(U_{n}^{\varepsilon,\eta})\otimes C(U_{n}^{\varepsilon,\eta})\\
 & u_{ij} & \mapsto & \sum_{k=1}^{n}u_{ik}\otimes u_{kj}
\end{array}.
\]
Note that if all non-diagonal entries of $\varepsilon$ and all entries
of $\eta$ are $1$, then we obtain the usual unitary group $U_{n}$
of degree $n$. If all entries of $\varepsilon$ and $\eta$ are $0$,
we get the free unitary group $(C(U_{n}^{+}),\Delta)$ introduced
by Shuzhou Wang \cite{vandalewang96universal}.
\begin{prop}
\label{prop:is cmqg}$U_{n}^{\varepsilon,\eta}$ is a compact matrix
quantum group.\end{prop}
\begin{proof}
It suffices to prove that $\Delta$ defines a $*$-homomorphism on
$C(U_{n}^{\varepsilon,\eta})$. By the universality of $C(U_{n}^{\varepsilon,\eta})$,
it remains to verify that the elements $u_{ij}'\coloneqq\Delta(u_{ij})$,
$1\leq i,j\leq n$ satisfy the relations in the definition of $C(U_{n}^{\varepsilon,\eta})$.
It is routine to see that the matrix $u'=[u_{ij}']$ and its conjugate
are unitary. The verification for $R^{\varepsilon}$ follows the same
pattern as in \cite{speicherweber16epsilinqg}, and we omit the details.
We are left with verifying the relations $R_{*}^{\eta}$.

In order to prove \eqref{eq:11} for $u_{ij}'$, we assume that $\eta_{ij}=\eta_{kl}=1$.
We have
\begin{align*}
u_{ik}'^{*}u_{jl}' & =\sum_{1\leq r,p\leq n}u_{ir}^{*}u_{jp}\otimes u_{rk}^{*}u_{pl}.
\end{align*}
By \eqref{eq:zero l} and \eqref{eq:zero r}, we see that $u_{ir}^{*}u_{jp}=u_{rk}^{*}u_{pl}=0$
if $p\neq r$ and $\eta_{rp}=0$. Hence together with \eqref{eq:xij yij}
the above equality can be rewritten as
\begin{align*}
u_{ik}'^{*}u_{jl}' & =\sum_{r,p:\eta_{rp}=1}u_{ir}^{*}u_{jp}\otimes u_{rk}^{*}u_{pl}+\sum_{r:\eta_{rr}=0}u_{ir}^{*}u_{jr}\otimes u_{rk}^{*}u_{rl}\\
 & =\sum_{r,p:\eta_{rp}=1}u_{ir}^{*}u_{jp}\otimes u_{rk}^{*}u_{pl}+|\{r:\eta_{rr}=0\}|X_{ij}\otimes Y_{kl}\\
 & =\sum_{r,p:\eta_{rp}=1}u_{jp}u_{ir}^{*}\otimes u_{pl}u_{rk}^{*}+|\{r:\eta_{rr}=0\}|X_{ij}\otimes Y_{kl},
\end{align*}
where the last equality follows from $\eqref{eq:11}$. Similarly we
have
\[
u_{jl}'u_{ik}'^{*}=\sum_{r,p:\eta_{rp}=1}u_{jp}u_{ir}^{*}\otimes u_{pl}u_{rk}^{*}+|\{r:\eta_{rr}=0\}|X_{ij}\otimes Y_{kl}.
\]
Thus we obtain $u_{ik}'^{*}u_{jl}'=u_{jl}'u_{ik}'^{*}$. 

For \eqref{eq:zero l}, assume that $\eta_{ij}=1,\eta_{kl}=0$ with
$k\neq l$. Then for any pair $(r,p)$ with $r\neq p$, or for $p=r$
with $\eta_{rr}=1$, either $u_{ir}^{*}u_{jp}=0$ or $u_{rk}^{*}u_{pl}=0$
according to \eqref{eq:zero l}. Hence we have
\begin{align*}
u_{ik}'^{*}u_{jl}' & =\sum_{1\leq r,p\leq n}u_{ir}^{*}u_{jp}\otimes u_{rk}^{*}u_{pl}=\sum_{r:\eta_{rr}=0}u_{ir}^{*}u_{jr}\otimes u_{rk}^{*}u_{rl}\\
 & =X_{ij}\otimes(\sum_{r:\eta_{rr}=0}u_{rk}^{*}u_{rl})=X_{ij}\otimes(0-\sum_{r:\eta_{rr}=1}u_{rk}^{*}u_{rl})=0,
\end{align*}
where the last equality follows from the fact that $u_{rk}^{*}u_{rl}=0$
for $\eta_{rr}=1$ according to \eqref{eq:zero l}. In the same way
we see that $u_{ik}'u_{jl}'^{*}=0$. The case $\eta_{ij}=0,\eta_{kl}=1,i\neq j$
is similar.

It remains to deal with the relations \eqref{eq:eq kk} and \eqref{eq:eq kk left}.
Assume that $\eta_{ij}=1$ and $\eta_{kk}=0$. We have

\begin{align*}
u_{ik}'^{*}u_{jk}' & =\sum_{r,p:\eta_{rp}=1}u_{ir}^{*}u_{jp}\otimes u_{rk}^{*}u_{pk}+\sum_{r,p:\eta_{rp}=0}u_{ir}^{*}u_{jp}\otimes u_{rk}^{*}u_{pk}\\
 & =\sum_{r,p:\eta_{rp}=1}u_{ir}^{*}u_{jp}\otimes X_{rp}+\sum_{r:\eta_{rr}=0}u_{ir}^{*}u_{jr}\otimes u_{rk}^{*}u_{rk} & \text{by \eqref{eq:zero l}}\\
 & =\sum_{r,p:\eta_{rp}=1}u_{ir}^{*}u_{jp}\otimes X_{rp}+X_{ij}\otimes(1-\sum_{r:\eta_{rr}=1}u_{rk}^{*}u_{rk})\\
 & =\sum_{r,p:\eta_{rp}=1}u_{jp}u_{ir}^{*}\otimes X_{rp}+X_{ij}\otimes(1-\sum_{r:\eta_{rr}=1}X_{rr}). & \text{by\eqref{eq:11}}
\end{align*}
Similarly we obtain
\[
u_{jk}'u_{ik}'^{*}=\sum_{r,p:\eta_{rp}=1}u_{jp}u_{ir}^{*}\otimes X_{rp}+X_{ij}\otimes(1-\sum_{r:\eta_{rr}=1}X_{rr}),
\]
which yields that $u_{ik}'^{*}u_{jk}'=u_{jk}'u_{ik}'^{*}$. Moreover,
we note that the right hand side of the above formula does not depend
on $k$. Therefore we see that for $\eta_{ll}=0$, we have
\[
u_{ik}'^{*}u_{jk}'=u_{il}'^{*}u_{jl}'=u_{jk}'u_{ik}'^{*}=u_{jl}'u_{il}'^{*},
\]
as desired. The case for \eqref{eq:eq kk left} is similar.
\end{proof}
Now we will prove that $U_{n}^{\varepsilon,\eta}$ is the quantum
symmetry group of $S_{\mathbb{C},\varepsilon,\eta}^{n-1}$. We refer
to \cite[Remark 4.10]{speicherweber16epsilinqg} for more explanation
on the notion of quantum symmetries in our setting.
\begin{thm}
Assume that $(\varepsilon,\eta)$ is regular. Then $U_{n}^{\varepsilon,\eta}$
is the quantum symmetry group of $S_{\mathbb{C},\varepsilon,\eta}^{n-1}$,
in the sense that $U_{n}^{\varepsilon,\eta}$ acts on $S_{\mathbb{C},\varepsilon,\eta}^{n-1}$
by homomorphisms
\[
\alpha,\beta:C(S_{\mathbb{C},\varepsilon,\eta}^{n-1})\to C(U_{n}^{\varepsilon,\eta})\otimes C(S_{\mathbb{C},\varepsilon,\eta}^{n-1}),
\]
\[
\alpha(x_{i})=\sum_{j}u_{ij}\otimes x_{j},\quad\beta(x_{i})=\sum_{k}u_{ki}\otimes x_{k},
\]
and for any compact matrix quantum group $\mathbb{G}$ acting on $S_{\mathbb{C},\varepsilon,\eta}^{n-1}$
in the above way, $\mathbb{G}$ is a compact matrix quantum subgroup
of $U_{n}^{\varepsilon,\eta}$.\end{thm}
\begin{proof}
Following the same pattern as in the proof of Proposition \ref{prop:is cmqg},
it is easy to check that the actions $\alpha$ and $\beta$ for $U_{n}^{\varepsilon,\eta}$
exist. In the following we only prove the maximality. In other words,
let $\mathbb{G}$ be another $n\times n$ compact matrix quantum group
with matrix coefficients $\{u_{ij}:1\leq i,j\leq n\}$, acting on
$S_{\mathbb{C},\varepsilon,\eta}^{n-1}$ via the actions
\[
\alpha',\beta':C(S_{\mathbb{C},\varepsilon,\eta}^{n-1})\to C(\mathbb{G})\otimes C(S_{\mathbb{C},\varepsilon,\eta}^{n-1}),
\]
\[
\alpha'(x_{i})=\sum_{j}u_{ij}\otimes x_{j},\quad\beta'(x_{i})=\sum_{k}u_{ki}\otimes x_{k}.
\]
We need to show that the unitary conditions and the relations $R^{\varepsilon}$
and $R_{*}^{\eta}$ hold for the generators $u_{ij}$. Let us verify
the relations $R_{*}^{\eta}$. To this end, for any $k,l$ with $k\neq l$,
we introduce the homomorphism 
\[
\pi_{kl}:C(S_{\mathbb{C},\varepsilon,\eta}^{n-1})\to C(S_{\mathbb{C},\tilde{\varepsilon},\tilde{\eta}}^{1}),\quad\pi_{kl}(x_{k})=x_{1},\pi_{kl}(x_{l})=x_{2},\pi_{kl}(x_{i})=0,\ i\neq k,l
\]
where $\tilde{\varepsilon}_{12}=\varepsilon_{kl}$, $\tilde{\eta}_{12}=\eta_{kl}$,
$\tilde{\eta}_{11}=\eta_{kk}$, $\tilde{\eta}_{22}=\eta_{ll}$.

Take $i,j$ with $\eta_{ij}=1$. In particular, for any $k,l$ with
$k\neq l$, we have
\[
(\mathrm{id}\otimes\pi_{kl})\circ\alpha'(x_{i}^{*}x_{j})=(\mathrm{id}\otimes\pi_{kl})\circ\alpha'(x_{j}x_{i}^{*}),
\]
which means
\begin{align*}
 & u_{ik}^{*}u_{jl}\otimes x_{1}^{*}x_{2}+u_{il}^{*}u_{jk}\otimes x_{2}^{*}x_{1}+u_{ik}^{*}u_{jk}\otimes x_{1}^{*}x_{1}+u_{il}^{*}u_{jl}\otimes x_{2}^{*}x_{2}\\
=\  & u_{jl}u_{ik}^{*}\otimes x_{2}x_{1}^{*}+u_{jk}u_{il}^{*}\otimes x_{1}x_{2}^{*}+u_{jk}u_{ik}^{*}\otimes x_{1}x_{1}^{*}+u_{jl}u_{il}^{*}\otimes x_{2}x_{2}^{*}.
\end{align*}
By definition, $x_{1}\mapsto-x_{1}$ gives a homomorphism of $C(S_{\mathbb{C},\tilde{\varepsilon},\tilde{\eta}}^{1})$,
so the above equality still hold when replacing $x_{1}$ by $-x_{1}$.
Combing these two equalities we obtain
\begin{equation}
u_{ik}^{*}u_{jl}\otimes x_{1}^{*}x_{2}+u_{il}^{*}u_{jk}\otimes x_{2}^{*}x_{1}=u_{jl}u_{ik}^{*}\otimes x_{2}x_{1}^{*}+u_{jk}u_{il}^{*}\otimes x_{1}x_{2}^{*},\label{eq:key eq1}
\end{equation}
\[
u_{ik}^{*}u_{jk}\otimes x_{1}^{*}x_{1}+u_{il}^{*}u_{jl}\otimes x_{2}^{*}x_{2}=u_{jk}u_{ik}^{*}\otimes x_{1}x_{1}^{*}+u_{jl}u_{il}^{*}\otimes x_{2}x_{2}^{*}.
\]
Recall that $x_{1}^{*}x_{1}+x_{2}^{*}x_{2}=x_{1}x_{1}^{*}+x_{2}x_{2}^{*}=1$.
The second equality above can be written as
\begin{equation}
u_{ik}^{*}u_{jk}-u_{jk}u_{ik}^{*}+(u_{il}^{*}u_{jl}-u_{ik}^{*}u_{jk})\otimes x_{2}^{*}x_{2}+(u_{jk}u_{ik}^{*}-u_{jl}u_{il}^{*})\otimes x_{2}x_{2}^{*}=0.\label{eq:key2}
\end{equation}
It is obvious to see that the unit element $1$ is linearly independent
from $\{x_{2}^{*}x_{2},x_{2}x_{2}^{*}\}$. Therefore we have
\begin{equation}
u_{ik}^{*}u_{jk}=u_{jk}u_{ik}^{*}.\label{eq:11bis pf}
\end{equation}

Now we prove \eqref{eq:11}. Assume $k\neq l$ and $\eta_{kl}=1$.
In this case we consider the torus $\mathbb{T}^{2}=\{(z_{1},z_{2})\in\mathbb{C}^{2}:|z_{1}|=|z_{2}|=1\}$,
and we have a homomorphism 
\[
\pi_{\mathbb{T}^{2}}:C(S_{\mathbb{C},\tilde{\varepsilon},\tilde{\eta}}^{1})\to C(\mathbb{T}^{2}),\quad\pi_{\mathbb{T}^{2}}(x_{1})=f_{1},\pi_{\mathbb{T}^{2}}(x_{2})=f_{2},
\]
where $f_{i}(z_{1},z_{2})=\frac{\sqrt{2}z_{i}}{2}$ ($i=1,2$) are
the coordinate functions. Applying $\mathrm{id}\otimes\pi_{\mathbb{T}^{2}}$
to \eqref{eq:key eq1}, we have 
\[
(u_{ik}^{*}u_{jl}-u_{jl}u_{ik}^{*})\otimes\bar{f}_{1}f_{2}-(u_{jk}u_{il}^{*}-u_{il}^{*}u_{jk})\otimes f_{1}\bar{f}_{2}=0.
\]
Note that $\bar{f}_{1}f_{2}$ and $f_{1}\bar{f}_{2}$ are linearly
independent. Therefore we obtain
\begin{equation}
u_{ik}^{*}u_{jl}=u_{jl}u_{ik}^{*},\quad u_{jk}u_{il}^{*}=u_{il}^{*}u_{jk}.\label{eq:11 pf}
\end{equation}
Together with \eqref{eq:11bis pf} we obtain \eqref{eq:11}. 

For \eqref{eq:zero l}, we assume $k\neq l$ and $\eta_{kl}=0$. If
$\varepsilon_{kl}=0$, we consider the full group C{*}-algebra $C^{*}(\mathbb{F}_{2})$
of the free group with two generators. Denote by $u_{1},u_{2}$ the
corresponding free unitary generators. Note that $u_{1}$ and $u_{2}$
are normal. We have a homomorphism
\begin{equation}
\pi_{\mathbb{F}_{2}}:C(S_{\mathbb{C},\tilde{\varepsilon},\tilde{\eta}}^{1})\to C^{*}(\mathbb{F}_{2}),\quad\pi_{\mathbb{F}_{2}}(x_{1})=\frac{\sqrt{2}}{2}u_{1},\pi_{\mathbb{F}_{2}}(x_{2})=\frac{\sqrt{2}}{2}u_{2}.\label{eq:pf free gp}
\end{equation}
Note that the elements $u_{1}^{*}u_{2},u_{1}u_{2}^{*},u_{2}^{*}u_{1},u_{2}u_{1}^{*}$
are linearly independent. Therefore, applying $\mathrm{id}\otimes\pi_{\mathbb{F}_{2}}$
to \eqref{eq:key eq1}, we have
\[
u_{ik}^{*}u_{jl}=u_{il}^{*}u_{jk}=u_{jl}u_{ik}^{*}=u_{jk}u_{il}^{*}=0.
\]
If $\varepsilon_{kl}=1$, then by our convention \eqref{eq:convention eta}
we have $\eta_{kk}=\eta_{ll}=0$. Then we have a homomorphism given
in Remark \ref{rem:normal}
\begin{equation}
\pi:C(S_{\mathbb{C},\tilde{\varepsilon},\tilde{\eta}}^{1})\to\mathbb{M}_{4}(\mathbb{C}),\quad\pi(x_{1})=a,\pi(x_{2})=b.\label{eq:pf 44}
\end{equation}
Here we see from Remark \ref{rem:normal} that the elements $a^{*}b,ab^{*},b^{*}a,ba^{*}$
are linearly independent. So applying $\mathrm{id}\otimes\pi$ to
\eqref{eq:key eq1}, we have
\[
u_{ik}^{*}u_{jl}=u_{il}^{*}u_{jk}=u_{jl}u_{ik}^{*}=u_{jk}u_{il}^{*}=0.
\]
Thus we proved \eqref{eq:zero l} as desired. By performing similar
computations for $\beta'$, we also obtain the relation \eqref{eq:zero r}.

For \eqref{eq:eq kk}, assume $\eta_{kk}=\eta_{ll}=0$. We will divide
the discussions into two cases: (a) there exists $1\leq m\leq n$
with $\eta_{mm}=0$ such that $\varepsilon_{km}\eta_{km}=0$ and $\varepsilon_{lm}\eta_{lm}=0$;
(b) otherwise. First suppose that (a) holds. 
If $\eta_{km}=0$, again we apply the above homomorphism $\mathrm{id}\otimes\pi$
to \eqref{eq:key2} with $l$ replaced by $m$. Note that the elements
$b^{*}b$ and $bb^{*}$ are linearly independent in $\mathbb{M}_{4}(\mathbb{C})$.
Hence we obtain
\[
u_{im}^{*}u_{jm}-u_{ik}^{*}u_{jk}=u_{jk}u_{ik}^{*}-u_{jm}u_{im}^{*}=0.
\]
Together with \eqref{eq:11bis pf} we obtain
\begin{equation}
u_{ik}^{*}u_{jk}=u_{im}^{*}u_{jm}=u_{jk}u_{ik}^{*}=u_{jm}u_{im}^{*}.\label{eq:ijkm}
\end{equation}
 If $\varepsilon_{km}=0$, it suffices to consider the noncommutative
sphere $S_{\mathbb{C},\tilde{\eta},\tilde{\varepsilon}}^{1}$ instead
of $S_{\mathbb{C},\tilde{\varepsilon},\tilde{\eta}}^{1}$, and replace
$x_{1}$ by its adjoint $x_{1}^{*}$ in \eqref{eq:pf 44}. Then the
same arguments as before yield the relation \eqref{eq:ijkm}. Similarly
we also get 
\[
u_{il}^{*}u_{jl}=u_{im}^{*}u_{jm}=u_{jl}u_{il}^{*}=u_{jm}u_{im}^{*}.
\]
Combining \eqref{eq:ijkm}, we obtain the desired relation \eqref{eq:eq kk}. 

Now suppose that (b) holds. In particular $\varepsilon_{kl}=\eta_{kl}=1$.
By the convention \eqref{eq:convention eta 2}, we take $m\neq m'$
with $\eta_{mm}=\eta_{m'm'}=0$ such that $\varepsilon_{km}\eta_{km}=0$
and $\varepsilon_{lm'}\eta_{lm'}=0$. In this case, instead of $\pi_{kl}$
we consider a homomorphism $\pi_{klmm'}:C(S_{\mathbb{C},\varepsilon,\eta}^{n-1})\to C(S_{\mathbb{C},\tilde{\varepsilon},\tilde{\eta}}^{3})$
with some suitable $(\tilde{\varepsilon},\tilde{\eta})$ which sends
$x_{k},x_{l},x_{m},x_{m'}$ to $x_{1},x_{2},x_{3},x_{4}$ in a similar
way. Then arguing similarly as in \eqref{eq:key2} and combining \eqref{eq:11bis pf},
we have
\begin{align*}
0 & =(u_{il}^{*}u_{jl}-u_{ik}^{*}u_{jk})\otimes x_{2}^{*}x_{2}+(u_{jk}u_{ik}^{*}-u_{jl}u_{il}^{*})\otimes x_{2}x_{2}^{*}\\
 & \qquad+(u_{im}^{*}u_{jm}-u_{ik}^{*}u_{jk})\otimes x_{3}^{*}x_{3}+(u_{jk}u_{ik}^{*}-u_{jm}u_{im}^{*})\otimes x_{3}x_{3}^{*}\\
\; & \qquad+(u_{im'}^{*}u_{jm'}-u_{ik}^{*}u_{jk})\otimes x_{4}^{*}x_{4}+(u_{jk}u_{ik}^{*}-u_{jm'}u_{im'}^{*})\otimes x_{4}x_{4}^{*}.
\end{align*}
We keep the notation of $a$ and $b$ in the representation $\pi$
as before. We consider the direct sum representation 
\[
\tilde{\pi}:C(S_{\mathbb{C},\tilde{\varepsilon},\tilde{\eta}}^{3})\to\mathbb{M}_{4}(\mathbb{C})\oplus\mathbb{M}_{4}(\mathbb{C}),\quad x_{1}\mapsto(a,0),x_{3}\mapsto(b,0),x_{2}\mapsto(0,a),x_{4}\mapsto(0,b).
\]
Then arguing as before by linear independence we obtain the desired
relation \eqref{eq:eq kk}

The relation \eqref{eq:eq kk left} follows similarly as above by
computations for $\beta'$.

The relations $R^{\varepsilon}$ can be proved by similar arguments
as in the proof of \cite[Theorem 4.7]{speicherweber16epsilinqg}.
The only non-obvious ingredient is that $x_{k}x_{l}$ and $x_{l}x_{k}$
are linearly independent for any choice of $\eta$ whenever $\varepsilon_{kl}=0$.
However this follows similarly as what we did for $\{x_{k}^{*}x_{l},x_{l}x_{k}^{*}\}$
in the case $\eta_{kl}=0$. Indeed, as is pointed out before, it suffices
to consider the noncommutative sphere $S_{\mathbb{C},\tilde{\eta},\tilde{\varepsilon}}^{1}$
instead of $S_{\mathbb{C},\tilde{\varepsilon},\tilde{\eta}}^{1}$,
and replace $x_{1}$ by its adjoint $x_{1}^{*}$ in \eqref{eq:pf free gp}
and \eqref{eq:pf 44}. We leave the details to the reader.

In the end we show that $u$ and $\bar{u}$ are unitary. We have
\begin{equation}
1\otimes1=\alpha'(\sum_{i=1}^{n}x_{i}^{*}x_{i})=\sum_{k,l=1}^{n}\sum_{i=1}^{n}u_{ik}^{*}u_{il}\otimes x_{k}^{*}x_{l}.\label{eq:unitary pf}
\end{equation}
Take an arbitrary $1\leq k\leq n$. Consider the circle $\mathbb{T}=\{z\in\mathbb{C}:|z|=1\}$
and the homomorphism
\[
\pi_{k}:C(S_{\mathbb{C},\varepsilon,\eta}^{n-1})\to C(\mathbb{T}),\quad\pi_{k}(x_{k})=f,\pi_{k}(x_{i})=0,\ i\neq k,
\]
where $f(z)=z,z\in\mathbb{T}$. Applying $\mathrm{id}\otimes\pi_{k}$
to the equality \eqref{eq:unitary pf} we obtain
\begin{equation}
\sum_{i=1}^{n}u_{ik}^{*}u_{ik}=1.\label{eq:unitary bis}
\end{equation}
Together with \eqref{eq:unitary pf} this also implies that 
\[
\sum_{1\leq k,l\leq n,k\neq l}\sum_{i=1}^{n}u_{ik}^{*}u_{il}\otimes x_{k}^{*}x_{l}=0.
\]
Take arbitrary $1\leq k,l\leq n$ with $k\neq l$ and consider the
homomorphism $\pi_{kl}$ introduced before. Applying $\mathrm{id}\otimes\pi_{kl}$
to the above equality, we get
\[
\sum_{i=1}^{n}u_{ik}^{*}u_{il}\otimes x_{1}^{*}x_{2}+\sum_{i=1}^{n}u_{il}^{*}u_{ik}\otimes x_{2}^{*}x_{1}=0.
\]
We have already seen that $x_{1}^{*}x_{2}$ and $x_{2}^{*}x_{1}$
are linearly independent in terms of the homomorphism $\pi_{\mathbb{T}^{2}}$.
Hence we deduce that
\[
\sum_{i=1}^{n}u_{ik}^{*}u_{il}=0.
\]
Together with \eqref{eq:unitary bis} we see that $u^{*}u=1$. Considering
the action on $x_{i}x_{i}^{*}$, we see also that $uu^{*}=1$. Hence
$u$ is unitary. Similar arguments for $\beta$ yield that $\bar{u}$
is unitary. Therefore the proof is finished.
\end{proof}

\begin{rem}
It is easy to see that there is a homomorphism 
\[
\phi:C(S_{\mathbb{C},\varepsilon,\eta}^{n-1})\to C(U_{n}^{\varepsilon,\eta}),\quad x_{i}\mapsto u_{i1}.
\]
Intuitively speaking, the sphere $S_{\mathbb{C},\varepsilon,\eta}^{n-1}$
can be viewed as a quantum space determined by the relations of the
first column of the quantum symmetry group which acts on it. A complete
theory towards this direction, in the setting of easy quantum groups,
has been recently developed by \cite{jungweber18qspace}. In general,
it is unclear whether the natural homomorphism in the form of $\phi$
is injective (see the comments in \cite[Section 2]{jungweber18qspace}).
Here we may provide a non-injective example in our setting of mixed
relations. More precisely, let $n=2$ and
\[
\varepsilon=\begin{bmatrix}0 & 0\\
0 & 0
\end{bmatrix},\quad\eta=\begin{bmatrix}0 & 1\\
1 & 0
\end{bmatrix},
\]
Then the natural homomorphism
\[
\phi:C(S_{\mathbb{C},\varepsilon,\eta}^{1})\to C(U_{2}^{\varepsilon,\eta}),\quad x_{i}\mapsto u_{i1},\ i=1,2
\]
is non-injective. Indeed, since $\eta_{11}=\eta_{22}=0$, by the unitary
condition of $u$ and $\bar{u}$ we have
\[
0=u_{11}u_{21}^{*}+u_{12}u_{22}^{*}=2X_{12}.
\]
In particular 
\[
\phi(x_{1}x_{2}^{*})=u_{11}u_{21}^{*}=X_{12}=0.
\]
However, we have $x_{1}x_{2}^{*}\neq0$. Indeed, consider the matrices
$a,b$ given in Remark \ref{rem:normal}. Then instead of $\pi$,
there is a representation 
\[
\pi':C(S_{\mathbb{C},\varepsilon,\eta}^{1})\to\mathbb{M}_{4}(\mathbb{C}),\quad\pi(x_{1})=a,\pi(x_{2})=b^{*},
\]
and 
\[
\pi(x_{1}x_{2}^{*})=ab=\begin{bmatrix}0\\
 & 0\\
 &  & 0\\
 &  &  & \frac{1}{2}
\end{bmatrix}\neq0.
\]
Therefore $\phi$ is not injective. 
\end{rem}

\section{Remarks on the orthogonal cases}

In this section we would like to discuss some related questions appeared
in \cite{speicherweber16epsilinqg}. In \cite{speicherweber16epsilinqg}
another version of commutation relations for quantum orthogonal groups
is proposed. More precisely, we consider the corresponding quantum
group given by the universal C{*}-algebra 
\[
C(\mathring{O}_{n}^{\varepsilon})=C^{*}(u_{ij},i,j=1,\ldots,n\mid u_{ij}=u_{ij}^{*},\ u\text{ is orthogonal, }\mathring{R}^{\varepsilon}\text{ holds}),
\]
where $\mathring{R}^{\varepsilon}$ denotes the relations
\[
u_{ik}u_{jl}=\begin{cases}
u_{jl}u_{ik} & \text{if \ensuremath{\varepsilon_{ij}=1} and \ensuremath{\varepsilon_{kl}=1}}\\
0 & \text{if \ensuremath{\varepsilon_{ij}=1} and \ensuremath{\varepsilon_{kl}=0}}\\
0 & \text{if \ensuremath{\varepsilon_{ij}=0} and \ensuremath{\varepsilon_{kl}=1}}
\end{cases}.
\]
We keep the convention that $\varepsilon=(\varepsilon_{ij} )_{i,j\in \{1,\ldots,n\} }$ is a symmetric matrice with $\varepsilon_{ij} \in \{0,1\}$ and $\varepsilon_{ii}=0$.  The quantum space on which $\mathring{O}_{n}^{\varepsilon}$ acts
maximally was left unsolved in \cite{speicherweber16epsilinqg}. In
the following we will briefly answer this question in terms of quantum
tuples of noncommutative spheres, inspired by \cite{jungweber18qspace}.
Consider the universal C{*}-algebra $C(\mathbb{X}_{n}^{\varepsilon})$
generated by $x_{ij},i,j=1,\ldots,n$ with relations 
\[
x_{ij}=x_{ij}^{*},\ \sum_{i}x_{ik}x_{il}=\delta_{kl},
\]
and
\[
x_{ik}x_{jl}=\begin{cases}
x_{jl}x_{ik} & \text{if \ensuremath{\varepsilon_{ij}=1} and \ensuremath{\varepsilon_{kl}=1}}\\
0 & \text{if \ensuremath{\varepsilon_{ij}=1} and \ensuremath{\varepsilon_{kl}=0}}\\
0 & \text{if \ensuremath{\varepsilon_{ij}=0} and \ensuremath{\varepsilon_{kl}=1}}
\end{cases}.
\]
We remark that in the case where $\varepsilon_{ij}=1$ for all $i\neq j$,
we see that $\mathring{O}_{n}^{\varepsilon}$ equals the classical
hyperoctahedral group $H_{n}$, and $\mathbb{X}_{n}^{\varepsilon}$
is simply the space of $n\times n$ orthogonal matrices with cubic
columns. Note that the set of cubic vectors $I_{n}\subset\mathbb{R}^{n}$
consists of the points $(0,\ldots,0,\pm1,0,\ldots,0)$ on each axis
of $\mathbb{R}^{n}$. It is well-know that $H_{n}$ is the symmetry
group of $I_{n}$. If $\varepsilon_{ij}=0$ for all $i,j$, then $\mathring{O}_{n}^{\varepsilon}$
is the free quantum orthogonal group $O_{n}^{+}$ introduced in \cite{vandalewang96universal}
(see also \cite{banicaspeicher09liberation}), and $\mathbb{X}_{n}^{\varepsilon}$
is the partition quantum space $X_{n,n}(\Pi)$ introduced in \cite{jungweber18qspace},
where in our setting $\Pi$ is the set of non-crossing pair partitions.
\begin{thm}
$\mathring{O}_{n}^{\varepsilon}$ is the quantum symmetry group of
$\mathbb{X}_{n}^{\varepsilon}$, in the sense that $\mathring{O}_{n}^{\varepsilon}$
acts on $\mathbb{X}_{n}^{\varepsilon}$ by homomorphisms 
\[
\alpha,\beta:C(\mathbb{X}_{n}^{\varepsilon})\to C(\mathring{O}_{n}^{\varepsilon})\otimes C(\mathbb{X}_{n}^{\varepsilon}),
\]
\[
\alpha(x_{ik})=\sum_{j}u_{ij}\otimes x_{jk},\quad\beta(x_{ik})=\sum_{j}u_{ji}\otimes x_{jk},
\]
and for any compact matrix quantum group $\mathbb{G}$ acting on $\mathbb{X}_{n}^{\varepsilon}$
in the above way, $\mathbb{G}$ is a compact matrix quantum subgroup
of $\mathring{O}_{n}^{\varepsilon}$.\end{thm}
\begin{proof}
We first check that the actions $\alpha$ and $\beta$ are well-defined.
It is a standard argument to see that $\alpha(x_{ij})=\alpha(x_{ij})^{*}$
and $\sum_{i}\alpha(x_{ik})\alpha(x_{il})=\delta_{kl}$ using the
orthogonal relations of $\mathring{O}_{n}^{\varepsilon}$. Also, according
to the relations $\mathring{R}^{\varepsilon}$, for $\varepsilon_{ij}=1,\varepsilon_{kl}=1$,
\[
\alpha(x_{ik})\alpha(x_{jl})=\sum_{p,q:\varepsilon_{pq}=1}u_{ip}u_{jq}\otimes x_{pk}x_{ql}=\sum_{p,q:\varepsilon_{pq}=1}u_{jq}u_{ip}\otimes x_{ql}x_{pk}=\alpha(x_{jl})\alpha(x_{ik}),
\]
and for $\varepsilon_{ij}=0,\varepsilon_{kl}=1$,
\[
\alpha(x_{ik})\alpha(x_{jl})=\sum_{p,q:\varepsilon_{pq}=0}u_{ip}u_{jq}\otimes x_{pk}x_{ql}=0,
\]
and for $\varepsilon_{ij}=1,\varepsilon_{kl}=0$,
\[
\alpha(x_{ik})\alpha(x_{jl})=\sum_{p,q:\varepsilon_{pq}=1}u_{ip}u_{jq}\otimes x_{pk}x_{ql}=0.
\]
Thus $\alpha$ is a well-defined homomorphism. Similarly we see that
the action $\beta$ exists as well. 

Now assume that $\mathbb{G}$ is an arbitrary compact matrix quantum
group acting on $\mathbb{X}_{n}^{\varepsilon}$ by homomorphisms
\[
\alpha',\beta':C(\mathbb{X}_{n}^{\varepsilon})\to C(\mathbb{G})\otimes C(\mathbb{X}_{n}^{\varepsilon}),
\]
\[
\alpha'(x_{ik})=\sum_{j}u_{ij}\otimes x_{jk},\quad\beta'(x_{ik})=\sum_{j}u_{ji}\otimes x_{jk}.
\]
Note that the diagonal C{*}-subalgebra generated by $\{x_{ii}:1\leq i\leq n\}$
in $C(\mathbb{X}_{n}^{\varepsilon})$ satisfies the relations $x_{ii}x_{jj}=x_{jj}x_{ii}$
for $\varepsilon_{ij}=1$. Restricting the homomorphisms $\alpha'$
and $\beta'$ to this subalgebra, the similar arguments as in \cite[Theorem 4.7]{speicherweber16epsilinqg}
yield that the generators $u_{ij}$ are self-adjoint, and the relation
$u_{ik}u_{jl}=u_{jl}u_{ik}$ for $\varepsilon_{ij}=\varepsilon_{kl}=1$
still holds, for which we omit the details. Now consider the case
$\varepsilon_{ij}=1,\varepsilon_{kl}=0$. We have a priori
\begin{equation}
0=\alpha'(x_{ik}x_{jk})=\sum_{p,q=1}^{n}u_{ip}u_{jq}\otimes x_{pk}x_{qk}=\sum_{p,q:\varepsilon_{pq}=0}u_{ip}u_{jq}\otimes x_{pk}x_{qk},\label{eq:o}
\end{equation}
where we applied the relations $x_{pk}x_{qk}=0$ for $\varepsilon_{pq}=1$
since $\varepsilon_{kk}=0$. If $k=l$, it is easy to see that there
exists a homomorphism $\pi_{1}:C(\mathbb{X}_{n}^{\varepsilon})\to\mathbb{C}$
such that $\pi_{1}(x_{kk})=1$ and $\pi_{1}(x_{k'k})=0$ for $k'\neq k$.
Applying the homomorphism $\mathrm{id}\otimes\pi_{1}$, the equality
\eqref{eq:o} yields 
\begin{equation}
u_{ik}u_{jk}=0.\label{eq:o pf 0}
\end{equation}
If $k\neq l$, we consider the surjective homomorphism $\pi_{O_{2}^{+}}:C(\mathbb{X}_{n}^{\varepsilon})\to C(O_{2}^{+})$
such that
\[
\mathrm{id}_{\mathbb{M}_{2}}\otimes\pi_{O_{2}^{+}}\left(\begin{bmatrix}x_{kk} & x_{kl}\\
x_{lk} & x_{ll}
\end{bmatrix}\right)=\begin{bmatrix}v_{11} & v_{12}\\
v_{21} & v_{22}
\end{bmatrix},
\]
where $v$ is the usual defining matrix of $C(O_{2}^{+})$. Applying
the homomorphism $\mathrm{id}\otimes\pi_{1}$, the equality \eqref{eq:o}
yields
\[
u_{ik}u_{jl}\otimes v_{11}v_{21}+u_{il}u_{jk}\otimes v_{21}v_{11}+u_{ik}u_{jk}\otimes v_{11}^{2}+u_{il}u_{jl}\otimes v_{21}^{2}=0.
\]
Moreover, together with \eqref{eq:o pf 0}, we see that
\[
u_{ik}u_{jl}\otimes v_{11}v_{21}+u_{il}u_{jk}\otimes v_{21}v_{11}=0.
\]
It is well-known that $v_{1j}v_{2j}$ and $v_{21}v_{11}$ are linearly
independent (see for example a simple matrix model in \cite[Theorem 3.9]{banica17spherenote}).
Hence we have
\[
u_{ik}u_{jl}=u_{il}u_{jk}=0.
\]
Continuing the similar arguments for the action $\beta'$, we obtain
completely the relations $\mathring{R}^{\varepsilon}$. 

Now the orthogonal relations for $\mathbb{G}$ follows easily. Note
that we have $\sum_{i}x_{ik}^{2}=1$ for all $k$. Therefore
\begin{align*}
1\otimes1 & =\alpha'(\sum_{i}x_{ik}^{2})=\sum_{p\neq q}\sum_{i}u_{ip}u_{iq}\otimes x_{pk}x_{qk}+\sum_{p}\sum_{i}u_{ip}^{2}\otimes x_{pk}^{2}\\
 & =\sum_{p\neq q:\varepsilon_{pq}=0}\sum_{i}u_{ip}u_{iq}\otimes x_{pk}x_{qk}+\sum_{p}\sum_{i}u_{ip}^{2}\otimes x_{pk}^{2}.
\end{align*}
Using the homomorphism $\pi_{1}$ as above, we deduce that $\sum_{i}u_{ik}^{2}=1$,
and hence
\[
\sum_{p\neq q:\varepsilon_{pq}=0}\sum_{i}u_{ip}u_{iq}\otimes x_{pk}x_{qk}=0.
\]
For $k\neq l$ with $\varepsilon_{kl}=0$, we use the homomorphism
$\pi_{O_{2}^{+}}$ as above and we obtain $\sum_{i}u_{ik}u_{il}=0$.
And for $\varepsilon_{kl}=1$, we see from the relation $\mathring{R}^{\varepsilon}$
that $\sum_{i}u_{ik}u_{il}=\sum_{i}0=0$. Repeating the similar arguments
with the action $\beta'$, we prove that $u$ is orthogonal. The proof
is complete.
\end{proof}

\subsection*{Acknowledgement }

The author would like to thank Stefan Jung and Moritz Weber for helpful
discussions. He would also like to thank the anonymous referee for
his careful reading and valuable suggestions on the preprint version.
The author was funded by the ERC Advanced Grant on Non-Commutative
Distributions in Free Probability, held by Roland Speicher.

\end{document}